\documentclass[a4paper,12pt, leqno]{amsart}
\usepackage{latexsym}
\usepackage{amssymb,amsfonts,amsmath,mathrsfs}
\begin{document}
\title{Gaps between zeros of the derivative of the Riemann $\xi$-function}
\author{H. M. Bui}
\address{Mathematical Institute, University of Oxford, OXFORD, OX1 3LB}
\email{hung.bui@maths.ox.ac.uk}
\thanks{The author is supported by an EPSRC Postdoctoral Fellowship}

\subjclass[2000]{11M26, 11M06}

\begin{abstract}
Assuming the Riemann hypothesis, we investigate the distribution of gaps between the zeros of $\xi'(s)$. We prove that a positive proportion of gaps are less than $0.796$ times the average spacing and, in the other direction, a positive proportion of gaps are greater than $1.18$ times the average spacing. We also exhibit the existence of infinitely many normalized gaps smaller (larger) than $0.7203$ ($1.5$, respectively).
\end{abstract}
\maketitle

\section{Introduction}

The Riemann $\xi$-function is defined by
\begin{equation*}
\xi(s)=\frac{s(s-1)}{2}\pi^{-s/2}\Gamma\big(\frac{s}{2}\big)\zeta(s),
\end{equation*}
where $\Gamma(s)$ is the Euler $\Gamma$-function and $\zeta(s)$ is the Riemann zeta-function. The $\xi$-function is an entire function with order $1$ and has a functional equation
\begin{equation*}
\xi(s)=\xi(1-s).
\end{equation*}
The zeros of $\xi(s)$ are identical to the complex zeros of the Riemann zeta-function. So if the Riemann hypothesis is true, all the zeros of $\xi(s)$, and so are the zeros of $\xi'(s)$, have real part $1/2$. Assuming the Riemann hypothesis, we write the zeros of $\xi'(s)$ as ${\scriptstyle{\frac{1}{2}}}+i\gamma_1$ (throughout the paper, the ordinates of the zeros of $\xi(s)$ will be denoted by $\gamma$, while those of $\xi'(s)$ will be denoted by $\gamma_1$). For $0<\gamma_1\leq\gamma_1'$ two consecutive ordinates of zeros, we define the normalized gap
\begin{equation*}
\delta(\gamma_1)=(\gamma_1'-\gamma_1)\frac{\log\gamma_1}{2\pi}.
\end{equation*}

The number of zeros of $\xi'(s)$ with ordinates in $[0,T]$ is $\frac{1}{2\pi}T\log T+O(T)$, so on average $\delta(\gamma_1)$ is $1$. In this paper, we are interested in the distribution of $\delta(\gamma_1)$. For a thorough discussion of the motivations of the problem, see [\textbf{\ref{FG}}]. It is expected that there exist arbitrarily small and large gaps between the zeros of $\xi'(s)$. That is to say
\begin{equation*}
\liminf_{\gamma_1}\delta(\gamma_1)=0\ \ \textrm{and}\ \limsup_{\gamma_1}\delta(\gamma_1)=\infty,
\end{equation*}
where $\gamma_1$ runs over all the ordinates of the zeros of $\xi'(s)$. We first establish

\newtheorem{theo}{Theorem}\begin{theo}
Assume RH. Then we have
\begin{equation*}
\liminf_{\gamma_1}\delta(\gamma_1)<0.7203\ \ \textrm{and}\ \ \limsup_{\gamma_1}\delta(\gamma_1)>1.5.
\end{equation*}
\end{theo}

\newtheorem{rem}{Remark}\begin{rem}
\emph{The existence of small and large gaps between the zeros of the Riemann zeta-function have been investigated by various authors [\textbf{\ref{MO}},\textbf{\ref{CGG1}},\textbf{\ref{CGG2}},\textbf{\ref{H}},\textbf{\ref{Ng}}]. The current best results, assuming the Riemann hypothesis, assert that $\liminf_{\gamma}\delta(\gamma)<0.5172$ and $\limsup_{\gamma}\delta(\gamma)>2.6306$, where $\gamma$ runs over the ordinates of the zeros of $\zeta(s)$. As discussed in [\textbf{\ref{FG}}], it is not surprising that these results are better than those obtained in our context.}
\end{rem}

We next define the upper and lower distribution functions
\begin{equation*}
D^{+}(\alpha)=\limsup_{T\rightarrow\infty}D(\alpha,T)\ \ \textrm{and}\ \ D^{-}(\alpha)=\liminf_{T\rightarrow\infty}D(\alpha,T),
\end{equation*}
where
\begin{equation*}
D(\alpha,T)=\bigg(\frac{1}{2\pi}T\log T\bigg)^{-1}\sum_{\substack{0<\gamma_1\leq T\\\delta(\gamma_1)\leq\alpha}}1.
\end{equation*}
Little is known about $D^{+}(\alpha)$ and $D^{-}(\alpha)$. It is expected that $D^{+}(\alpha)=D^{-}(\alpha)\ (=D(\alpha))$ for all $\alpha$ and that $D(0)=0$, $D(\alpha)<1$ for all $\alpha$, and $D(\alpha)$ is continuous. In a recent paper, by developing an analogue of Montgomery' result [\textbf{\ref{M}}] for the pair correlation of the zeros of $\xi'(s)$, Farmer and Gonek [\textbf{\ref{FG}}] proved that
\begin{equation*}
D^{-}(0.91)>0\ \ \textrm{and}\ \ D^{-}(1)>0.035.
\end{equation*}
That means that a positive proportion of gaps between the zeros of $\xi'(s)$ are less than $0.91$ times the average spacing, and more than $3.5\%$ of the normalized neighbour gaps are smaller than average. We slightly improve upon their first statement and show that

\begin{theo}
Assume RH. Then we have
\begin{equation*}
D^{-}(0.796)>0\ \ \textrm{and}\ \ D^{+}(1.18)<1.
\end{equation*}
\end{theo}

\begin{rem}
\emph{It is possible that our theorems can be improved by using some other choices of some coefficients. However, we have made no serious attempt to obtain the optimal results given by this method.}
\end{rem}

In the context of the Riemann zeta-function, it is also known that a positive proportion of normalized gaps between the zeros of $\zeta(s)$ are less (more) than $0.6878$ ($1.4843$, respectively) [\textbf{\ref{So}}]. Other results involving the zeros of the higher derivatives of the Riemann $\xi$-function are also proved in [\textbf{\ref{B}}].

The paper is organized as follows. In the next section, we sketch the idea to attack our theorems. Section 3 contains all the necessary lemmas. We prove Theorem 1 in Section 4. The final section is devoted to Theorem 2.

\section{Initial manipulations}

Throughout the article, we assume the Riemann hypothesis. We also assume that $y=(T/2\pi)^{\theta}$, where $0<\theta<1/2$, and $r\geq1$. We denote $L=\log\frac{T}{2\pi}$ and define
\begin{equation*}
h_k(\alpha,M)=\frac{\int_{-\pi\alpha/L}^{\pi\alpha/L}\sum_{T<\gamma_1\leq2T}|M(\frac{1}{2}+i\gamma_1+it)|^{2k}dt}{\int_{T}^{2T}|M(\frac{1}{2}+it)|^{2k}dt},
\end{equation*}
where 
\begin{equation*}
M(s)=\sum_{n\leq y}\frac{a(n)f(\frac{\log y/n}{\log y})}{n^s}
\end{equation*}
for some arithmetic function $a(n)$ and smooth function $f(x)$. We will see later that in order to prove Theorem 1 we would like to choose $a(n)=d_r(n)$ for the large gaps, and $a(n)=\mu_r(n)$ for the small gaps, where $d_r(n)$ and $\mu_r(n)$ are the coefficients of $n^{-s}$ in the Dirichlet series of $\zeta(s)^r$ and $\zeta(s)^{-r}$, respectively:
\begin{equation*}
\zeta(s)^r=\sum_{n=1}^{\infty}\frac{d_r(n)}{n^s}\ \textrm{and}\ \zeta(s)^{-r}=\sum_{n=1}^{\infty}\frac{\mu_r(n)}{n^s}\qquad(\sigma>1).
\end{equation*}
In the case of Theorem 2, the coefficients $a(n)$ are chosen to be supported on $1$ and the primes.

Theorem 1 is based on the following idea of Mueller [\textbf{\ref{Mu}}]. Given that
\begin{equation*}
\liminf_{\gamma_1}\delta(\gamma_1)=\mu\ \ \textrm{and}\ \ \limsup_{\gamma_1}\delta(\gamma_1)=\lambda.
\end{equation*}
It is easy to see that
\begin{eqnarray*}
\int_{-\pi\mu/L}^{\pi\mu/L}\sum_{T<\gamma_1\leq2T}|M({\scriptstyle{\frac{1}{2}}}+i\gamma_1+it)|^{2k}dt&\leq&(1+o(1))\int_{T}^{2T}|M({\scriptstyle{\frac{1}{2}}}+it)|^{2k}dt\\ &&\qquad\leq\int_{-\pi\lambda/L}^{\pi\lambda/L}\sum_{T<\gamma_1\leq2T}|M({\scriptstyle{\frac{1}{2}}}+i\gamma_1+it)|^{2k}dt.
\end{eqnarray*}
So $h_k(\mu,M)\leq1+o(1)\leq h_k(\lambda,M)$. Clearly, $h_k(\alpha,M)$ is monotonically increasing with respect to $\alpha$. Therefore, if $h_k(\alpha,M)<1$ for some choice of $\alpha$ and $M$, then $\alpha<\lambda$. Similarly, if $h_k(\alpha,M)>1$ then $\alpha>\mu$. Thus it suffices to show that
\begin{equation}\label{8}
h_1(1.5,M_1)<1\ \ \textrm{and}\ \ h_1(0.7203,M_2)>1,
\end{equation}
for some $M_1$ and $M_2$.

To attack Theorem 2, we follow the setting of [\textbf{\ref{CGGH}}]. For $\gamma_{1}^{\dagger}\leq\gamma_1\leq\gamma_1'$ three consecutive ordinates of zeros of $\xi'(s)$, let
\begin{equation*}
\delta^{+}(\gamma_1)=(\gamma_1'-\gamma_1)L/2\pi,\ \textrm{and}\ \delta^{-}(\gamma_1)=(\gamma_1-\gamma_1^{\dagger})L/2\pi.
\end{equation*}
Also let
\begin{equation*}
\delta_{0}(\gamma_1)=\min\{\delta^{+}(\gamma_1),\delta^{-}(\gamma_1)\},\ \textrm{and}\ \delta_{1}(\gamma_1)=\max\{\delta^{+}(\gamma_1),\delta^{-}(\gamma_1)\}.
\end{equation*}

We first establish the formula for the large gaps. We have, up to an error term of size $O(T^{1-\varepsilon})$,
\begin{eqnarray*}
\int_{T}^{2T}|M({\scriptstyle{\frac{1}{2}}}+it)|^{2}dt&=&\sum_{T<\gamma_1\leq 2T}\int_{\gamma_1-\pi\delta^{-}(\gamma_1)/L}^{\gamma_1+\pi\delta^{+}(\gamma_1)/L}|M({\scriptstyle{\frac{1}{2}}}+it)|^{2}dt\nonumber\\
&&\!\!\!\!\!\!\!\!\!\!\!\!\!\!\!\!\!\!\!\!\!\!\!\!\!\!\!\!\!\!\!\!\!\!\!\!\!\!\!\!\!\!\!\!\!\!\!\!\!\!\!\!\!\leq\sum_{\substack{T<\gamma_1\leq2T\\\delta_1(\gamma_1)\leq\lambda}}\int_{-\pi\lambda/L}^{\pi\lambda/L}|M({\scriptstyle{\frac{1}{2}}}+i(\gamma_1+t))|^{2}dt+\sum_{\substack{T<\gamma_1\leq2T\\\delta_1(\gamma_1)>\lambda}}\int_{-\pi\delta^{-}(\gamma_1)/L}^{\pi\delta^{+}(\gamma_1)/L}|M({\scriptstyle{\frac{1}{2}}}+i(\gamma_1+t))|^{2}dt\nonumber\\
&&\!\!\!\!\!\!\!\!\!\!\!\!\!\!\!\!\!\!\!\!\!\!\!\!\!\!\!\!\!\!\!\!\!\!\!\!\!\!\!\!\!\!\!\!\!\!\!\!\!\!\!\!\!\leq h_1(\lambda,M)\int_{T}^{2T}|M({\scriptstyle{\frac{1}{2}}}+it)|^{2}dt+\sum_{\substack{T<\gamma_1\leq2T\\\delta_1(\gamma_1)>\lambda}}\int_{-\pi\delta^{-}(\gamma_1)/L}^{\pi\delta^{+}(\gamma_1)/L}|M({\scriptstyle{\frac{1}{2}}}+i(\gamma_1+t))|^{2}dt.
\end{eqnarray*}
Hence
\begin{equation*}
(1-h_1(\lambda,M))\int_{T}^{2T}|M({\scriptstyle{\frac{1}{2}}}+it)|^{2}dt+O(T^{1-\varepsilon})\leq\sum_{\substack{T<\gamma_1\leq2T\\\delta_1(\gamma_1)>\lambda}}\int_{-\pi\delta^{-}(\gamma_1)/L}^{\pi\delta^{+}(\gamma_1)/L}|M({\scriptstyle{\frac{1}{2}}}+i(\gamma_1+t))|^{2}dt.
\end{equation*}
Using Cauchy's inequality, the right hand side is bounded by
\begin{equation*}
\bigg(\frac{2\pi}{L}\bigg)^{\frac{1}{2}}\bigg(\sum_{\substack{T<\gamma_1\leq2T\\\delta_1(\gamma_1)>\lambda}}1\bigg)^{\frac{1}{4}}\bigg(\sum_{T<\gamma_1\leq2T}\delta(\gamma_1)^2\bigg)^{\frac{1}{4}}\bigg(\int_{T}^{2T}|M({\scriptstyle{\frac{1}{2}}}+it)|^{4}dt\bigg)^{\frac{1}{2}}.
\end{equation*}
Thus, if $h_1(\lambda,M)<1$,
\begin{equation}\label{2}
\sum_{\substack{T<\gamma_1\leq2T\\\delta_1(\gamma_1)>\lambda}}1\geq\frac{(1-h_1(\lambda,M))^4\big(\int_{T}^{2T}|M({\scriptstyle{\frac{1}{2}}}+it)|^{2}dt\big)^4}{4\pi^2\big(\sum_{T<\gamma_1\leq2T}\delta(\gamma_1)^2\big)\big(\int_{T}^{2T}|M({\scriptstyle{\frac{1}{2}}}+it)|^{4}dt\big)^2}L^2+o(1).
\end{equation}

The small gaps can be treated in a similar way. Up to an error term of size $O(T^{1-\varepsilon})$, we have
\begin{eqnarray*}
\int_{T}^{2T}|M({\scriptstyle{\frac{1}{2}}}+it)|^{2}dt&=&\sum_{T<\gamma_1\leq 2T}\int_{\gamma_1-\pi\delta^{-}(\gamma_1)/L}^{\gamma_1+\pi\delta^{+}(\gamma_1)/L}|M({\scriptstyle{\frac{1}{2}}}+it)|^{2}dt\nonumber\\
&&\!\!\!\!\!\!\!\!\!\!\!\!\!\!\!\!\!\!\!\!\!\!\!\!\!\!\!\!\!\!\!\!\!\!\!\!\!\!\!\!\!\!\!\!\!\!\!\!\!\!\!\!\!\geq\sum_{\substack{T<\gamma_1\leq2T\\\delta_0(\gamma_1)<\mu}}\int_{-\pi\delta^{-}(\gamma_1)/L}^{\pi\delta^{+}(\gamma_1)/L}|M({\scriptstyle{\frac{1}{2}}}+i(\gamma_1+t))|^{2}dt+\sum_{\substack{T<\gamma_1\leq2T\\\delta_0(\gamma_1)\geq\mu}}\int_{-\pi\mu/L}^{\pi\mu/L}|M({\scriptstyle{\frac{1}{2}}}+i(\gamma_1+t))|^{2}dt\nonumber\\
&&\!\!\!\!\!\!\!\!\!\!\!\!\!\!\!\!\!\!\!\!\!\!\!\!\!\!\!\!\!\!\!\!\!\!\!\!\!\!\!\!\!\!\!\!\!\!\!\!\!\!\!\!\!\geq h_1(\mu,M)\int_{T}^{2T}|M({\scriptstyle{\frac{1}{2}}}+it)|^{2}dt\nonumber\\
&&\!\!\!\!\!\!\!\!\!\!\!\!\!\!\!\!\!\!\!\!\!\!\!\!\!\!\!\!-\sum_{\substack{T<\gamma_1\leq2T\\\delta_0(\gamma_1)<\mu}}\int_{\pi\delta_0(\gamma_1)/L}^{\pi\mu/L}\big(|M({\scriptstyle{\frac{1}{2}}}+i(\gamma_1+t))|^{2}+|M({\scriptstyle{\frac{1}{2}}}+i(\gamma_1-t))|^{2}\big)dt.
\end{eqnarray*}
Hence
\begin{eqnarray*}
(h_1(\mu,M)-1)\int_{T}^{2T}|M({\scriptstyle{\frac{1}{2}}}+it)|^{2}dt+O(T^{1-\varepsilon})&\leq&\nonumber\\
&&\!\!\!\!\!\!\!\!\!\!\!\!\!\!\!\!\!\!\!\!\!\!\!\!\!\!\!\!\!\!\!\!\!\!\!\!\!\!\!\!\!\!\!\!\!\!\!\!\!\!\!\!\!\!\!\!\!\!\!\!\!\!\!\!\!\!\!\!\!\!\!\!\!\!\!\!\!\!\!\!\!\!\!\!\!\!\!\!\!\!\!\!\!\!\!\!\!\!\!\!\!\!\!\!\!\!\sum_{\substack{T<\gamma_1\leq2T\\\delta_0(\gamma_1)<\mu}}\int_{\pi\delta_0(\gamma_1)/L}^{\pi\mu/L}\big(|M({\scriptstyle{\frac{1}{2}}}+i(\gamma_1+t))|^{2}+|M({\scriptstyle{\frac{1}{2}}}+i(\gamma_1-t))|^{2}\big)dt.
\end{eqnarray*}
Using Cauchy's inequality, the right hand side is 
\begin{eqnarray*}
&\leq&\bigg(\frac{2\pi\mu}{L}\bigg)^{\frac{1}{2}}\bigg(\sum_{\substack{T<\gamma_1\leq2T\\\delta_0(\gamma_1)<\mu}}1\bigg)^{\frac{1}{2}}\nonumber\\
&&\qquad\bigg(\sum_{\substack{T<\gamma_1\leq2T\\\delta_0(\gamma_1)<\mu}}\int_{\pi\delta_0(\gamma_1)/L}^{\pi\mu/L}\big(|M({\scriptstyle{\frac{1}{2}}}+i(\gamma_1+t))|^{4}+|M({\scriptstyle{\frac{1}{2}}}+i(\gamma_1-t))|^{4}\big)dt\bigg)^{\frac{1}{2}}\nonumber\\
&\leq&\bigg(\frac{2\pi\mu}{L}\bigg)^{\frac{1}{2}}\bigg(\sum_{\substack{T<\gamma_1\leq2T\\\delta_0(\gamma_1)<\mu}}1\bigg)^{\frac{1}{2}}\bigg(h_2(\mu,M)\int_{T}^{2T}|M({\scriptstyle{\frac{1}{2}}}+it)|^{4}dt\bigg)^{\frac{1}{2}}.
\end{eqnarray*}
Thus, if $h_1(\mu,M)>1$,
\begin{equation}\label{3}
\sum_{\substack{T<\gamma_1\leq2T\\\delta_0(\gamma_1)<\mu}}1\geq\frac{(h_1(\mu,M)-1)^2\big(\int_{T}^{2T}|M({\scriptstyle{\frac{1}{2}}}+it)|^{2}dt\big)^2}{2\pi\mu h_2(\mu,M)\int_{T}^{2T}|M({\scriptstyle{\frac{1}{2}}}+it)|^{4}dt}L+o(1).
\end{equation}

In the rest of the paper, we will illustrate the inequality $h_1(\lambda,M_1)<1<h_1(\mu,M_2)$ for suitable $\lambda$, $\mu$, $M_1$, $M_2$, and evaluate the expressions in \eqref{2} and \eqref{3}.

\begin{rem}
\emph{To exhibit the existence of positive proportion of large and small gaps, we need to show that the orders of magnitude of the right hand sides in \eqref{2} and \eqref{3} are $T\log T$. It will be clear later in our proof that this requires
\begin{equation*}
\bigg(\int_{T}^{2T}|M({\scriptstyle{\frac{1}{2}}}+it)|^{2}dt\bigg)^2\asymp T\int_{T}^{2T}|M({\scriptstyle{\frac{1}{2}}}+it)|^{4}dt.
\end{equation*}
This condition restricts the choice of our Dirichlet polynomial $M$.}
\end{rem}

\section{Auxiliary lemmas}

We need various lemmas concerning divisor sums and other divisor-like sums. We first introduce some notations which we will use throughout. Let $A_r(n)=A_r(n,1)$, where
\begin{equation*}
A_{r}(n,s):=\prod_{p^\lambda||n}\frac{\sum_{j=0}^{\infty}d_r(p^j)d_r(p^{j+\lambda})p^{-js}}{\sum_{j=0}^{\infty}d_r(p^j)^2p^{-js}}\qquad(\sigma>1).
\end{equation*}
We define 
\begin{equation*}
F_\tau(n)=\prod_{p|n}(1+O(p^{-\tau})),
\end{equation*}
for $\tau>0$ and the constant in the $O$-term is implicit and independent of $\tau$. We note that
\begin{equation*}
A_{r}(n,s)\ll d_r(n)F_\tau(n)\qquad(\sigma\geq\tau>0).
\end{equation*}

\newtheorem{lemm}{Lemma}\begin{lemm}
We have
\begin{equation*}
\sum_{n\leq y}\frac{d_{r}(n)^2}{n}=\frac{a_{r}(\log y)^{r^2}}{\Gamma(r^{2}+1)}+O((\log y)^{r^2-1}),
\end{equation*}
where
\begin{equation*}
a_{r}=\prod_{p}\bigg(\bigg(1-\frac{1}{p}\bigg)^{r^{2}}\sum_{n\geq0}\frac{d_{r}(p^n)^2}{p^{n}}\bigg).
\end{equation*}
\end{lemm}
\begin{proof}
The proof of this fact is standard.
\end{proof}

\begin{lemm}
There exists an absolute constant $\tau_0>0$ such that
\begin{equation*}
\sum_{m\leq x}\frac{d_r(m)d_r(mn)}{m}=\frac{a_rA_r(n)(\log x)^{r^2}}{\Gamma(r^2+1)}+O(d_r(n)F_{\tau_0}(n)(\log x)^{r^2-1}).
\end{equation*}
\end{lemm}
\begin{proof}
We note that $G(m)=g(mn)/g(n)$ is a multiplicative function whenever $g$ is (provided that $g(n)\ne0$). Hence
\begin{eqnarray*}
\sum_{m=1}^{\infty}\frac{d_r(m)d_r(mn)}{m^s}&=&\prod_{p\nmid n}\bigg(\sum_{j=0}^{\infty}\frac{d_r(p^j)^2}{p^{js}}\bigg)\prod_{p^\lambda||n}\bigg(\sum_{j=0}^{\infty}\frac{d_r(p^j)d_r(p^{j+\lambda})}{p^{js}}\bigg)\\
&=&\zeta(s)^{r^2}A_r(n,s)\prod_{p}\bigg(\bigg(1-\frac{1}{p}\bigg)^{r^{2}}\sum_{j=0}^{\infty}\frac{d_{r}(p^j)^2}{p^{js}}\bigg),
\end{eqnarray*}
for $\sigma>1$. The lemma follows by applying Theorem 2 of Selberg [\textbf{\ref{S}}].
\end{proof}

The following lemma is an easy consequence of Lemma 1.

\begin{lemm}
Given that
\begin{equation*}
M(s)=\sum_{n\leq y}\frac{d_r(n)f(\frac{\log y/n}{\log y})}{n^s}.
\end{equation*}
Then we have
\begin{equation*}
\int_{T}^{2T}|M({\scriptstyle{\frac{1}{2}}}+it)|^2dt\sim \frac{a_{r}T(\log y)^{r^2}}{\Gamma(r^{2})}\int_{0}^{1}(1-x)^{r^{2}-1}f(x)^2dx.
\end{equation*}
\end{lemm}
\begin{proof}
Using Montgomery \& Vaughan's mean value theorem [\textbf{\ref{MV1}}] we have
\begin{equation*}
\int_{T}^{2T}|M({\scriptstyle{\frac{1}{2}}}+it)|^2dt\sim \sum_{n\leq y}\frac{d_r(n)^2f(\frac{\log y/n}{\log y})^2}{n}.
\end{equation*}
The lemma follows from Lemma 1 and Stieltjes integration.
\end{proof}

In order to estimate the nominator of $h_k(\alpha,M)$ we will use Cauchy's residue theorem. To this end, we need a Dirichlet series for $\xi''/\xi'(s)$. From the definition of the Riemann $\xi$-function we have
\begin{equation*}
\frac{\xi'}{\xi}(s)=L(s)+\frac{\zeta'}{\zeta}(s),
\end{equation*}
where
\begin{equation}\label{9}
L(s)=\frac{1}{s}+\frac{1}{s-1}-\frac{\log\pi}{2}+\frac{1}{2}\frac{\Gamma'}{\Gamma}\bigg(\frac{s}{2}\bigg).
\end{equation}
We cite a lemma of Farmer \& Gonek [\textbf{\ref{FG}}].

\begin{lemm}
For $\sigma=1+L^{-1}$, $T<\Im{s}\leq 2T$ and $K$ a large positive integer we have
\begin{equation*}
\frac{\xi''}{\xi'}(s)=\frac{L}{2}+\sum_{n=1}^{\infty}\frac{a_{K}(n,s)}{n^s}+O(1+L^{2}2^{-K}).
\end{equation*}
Here we have written
\begin{equation*}
a_K(n,s)=\sum_{k=0}^{K}\frac{\alpha_k(n)}{L(s)^k},
\end{equation*}
where
\begin{equation*}
\alpha_k(n)=\left\{ \begin{array}{ll}
-\Lambda(n) &\qquad \textrm{if $k=0$}\\
\Lambda_{k-1}*\Lambda\log(n) &\qquad \textrm{if $k\geq1$.}
\end{array} \right.
\end{equation*}
\end{lemm}

\begin{rem}
The function $\Lambda_j$ for $j\geq0$ is the $j$-fold convolution of the von Mangoldt function, defined by
\begin{equation*}
\bigg(-\frac{\zeta'}{\zeta}(s)\bigg)^j=\sum_{n=1}^{\infty}\frac{\Lambda_j(n)}{n^s}.
\end{equation*}
\end{rem}

\begin{lemm}
Let $\sigma=1+L^{-1}$. Then for $x>0$ we have
\begin{equation*}
\frac{1}{2\pi i}\int_{\sigma+iT}^{\sigma+i2T}\frac{x^{i\tau}}{L(s)^k}ds=\left\{ \begin{array}{ll}
\frac{T}{2\pi}(\frac{L}{2})^{-k}(1+O_K(L^{-1})) &\qquad \textrm{if $x=1$}\\
O_K\big(\frac{1}{|\log x|}\big) &\qquad \textrm{otherwise,}
\end{array} \right.
\end{equation*}
where $s=\sigma+i\tau$.
\end{lemm}
\begin{proof}
We deduce from \eqref{9} that
\begin{equation}\label{10}
L(s)=\frac{1}{2}\log\frac{s}{2\pi}+O\bigg(\frac{1}{|s|+2}\bigg),\ \textrm{and}\ L'(s)\ll\frac{1}{|s|+2}.
\end{equation}
Hence
\begin{equation*}
L(s)^{-k}=(1+O((\log\tau)^{-1}))\bigg(\frac{1}{2}\log\frac{s}{2\pi}\bigg)^{-k}.
\end{equation*}
The case $x=1$ follows immediately.

For $x\ne1$, integration by parts leads to
\begin{equation*}
\frac{x^{it}}{i\log xL(s)^k}\bigg|_{T}^{2T}+\frac{k}{i\log x}\int_{T}^{2T}\frac{x^{it}L'(s)}{L(s)^{k+1}}dt.
\end{equation*}
Using \eqref{10}, this is 
\begin{equation*}
\ll\frac{1}{|\log x|}\bigg(1+k\int_{T}^{2T}\frac{dt}{\tau(1/2\log\tau/2\pi)^{k+1}}\bigg)\ll\frac{1}{|\log x|}.
\end{equation*}
The proof is complete.
\end{proof}

The next two lemmas concern various sums involving $\alpha_k(n)$.

\begin{lemm}
For $\alpha_k(n)$ defined as in the previous lemma we have
\begin{equation*}
T_k(x)=\sum_{n\leq x}\frac{\alpha_k(n)A_r(n)}{n}=\left\{ \begin{array}{ll}
-r\log x+O(1) &\qquad \textrm{if $k=0$}\\
\frac{r^{k}(\log x)^{k+1}}{(k+1)!}+O((\log x)^{k}) &\qquad \textrm{if $k\geq1$.}
\end{array} \right.
\end{equation*}
As a consequence, for $x\ll T^{1/2}$ we obtain
\begin{equation*}
\sum_{n\leq x}\sum_{k=0}^{K}\bigg(\frac{L}{2}\bigg)^{-k}\frac{\alpha_k(n)A_r(n)}{n}=-r\log x+\sum_{k=1}^{K}\bigg(\frac{L}{2}\bigg)^{-k}\frac{r^{k}(\log x)^{k+1}}{(k+1)!}+O(K).
\end{equation*}
\end{lemm}
\begin{proof}
We will just prove the first statement. We need to separate the cases $k=0$, $k=1$ and $k\geq2$. We have
\begin{eqnarray*}
T_0(x)&=&-\sum_{n\leq x}\frac{\Lambda(n)A_r(n)}{n}=-\sum_{p^\lambda\leq x}\frac{(\log p)A_r(p^\lambda)}{p^\lambda}=-\sum_{p\leq x}\frac{(\log p)A_r(p)}{p}+O(1)\nonumber\\
&=&-r\sum_{p\leq x}\frac{\log p}{p}+O(1)=-r\log x+O(1).
\end{eqnarray*}
Similarly for $k=1$,
\begin{eqnarray*}
T_{1}(x)&=&\sum_{n\leq x}\frac{(\Lambda\log)(n)A_r(n)}{n}=\sum_{p^\lambda\leq x}\frac{\lambda(\log p)^2A_r(p^\lambda)}{p^\lambda}\nonumber\\
&=&\sum_{p\leq x}\frac{(\log p)^2A_r(p)}{p}+O(1)=\frac{r(\log x)^2}{2}+O(1).
\end{eqnarray*}
Now for $k\geq2$ we have
\begin{eqnarray}\label{5}
T_{k}(x)&=&\sum_{n\leq x}\frac{(\Lambda_{k-1}*\Lambda\log)(n)A_r(n)}{n}=\sum_{n\leq x}\sum_{p^\lambda|n}\frac{\lambda(\log p)^2\Lambda_{k-1}(\frac{n}{p^\lambda})A_r(n)}{n}\nonumber\\
&=&\sum_{n\leq x}\sum_{p|n}\frac{(\log p)^2\Lambda_{k-1}(\frac{n}{p})A_r(n)}{n}+O((\log x)^{k})\nonumber\\
&=&\sum_{p\leq x}\frac{(\log p)^2}{p}\sum_{n\leq x/p}\frac{\Lambda_{k-1}(n)A_r(pn)}{n}+O((\log x)^{k}).
\end{eqnarray}

We are going to prove by induction that there exists an absolute constant $\tau_0$ such that for $k\geq1$
\begin{equation}\label{4}
\sum_{n\leq x}\frac{\Lambda_{k}(n)A_r(mn)}{n}=\frac{r^{k}A_r(m)(\log x)^{k}}{k!}+O(d_r(m)F_{\tau_0}(m)(\log x)^{k-1}).
\end{equation}
For the base case we have
\begin{eqnarray*}
\sum_{n\leq x}\frac{\Lambda(n)A_r(mn)}{n}&=&\sum_{p^\lambda\leq x}\frac{(\log p)A_r(mp^\lambda)}{p^\lambda}=\sum_{p\leq x}\frac{(\log p)A_r(mp)}{p}+O(d_r(m)F_{\tau}(m))\nonumber\\
&=&rA_r(m)\log x+O(d_r(m)F_{\tau}(m)).
\end{eqnarray*}
Now for $k\geq1$,
\begin{eqnarray*}
\sum_{n\leq x}\frac{\Lambda_{k+1}(n)A_r(mn)}{n}&=&\sum_{n\leq x}\sum_{p^\lambda|n}\frac{(\log p)\Lambda_{k}(\frac{n}{p^\lambda})A_r(mn)}{n}\nonumber\\
&&\!\!\!\!\!\!\!\!\!\!\!\!\!\!\!\!\!\!\!\!\!\!\!\!\!\!\!\!\!\!\!\!\!\!\!\!\!\!\!\!\!\!\!\!=\sum_{n\leq x}\sum_{p|n}\frac{\log p\Lambda_{k}(\frac{n}{p})A_r(mn)}{n}+O(d_r(m)F_{\tau_0}(m)(\log x)^{k})\nonumber\\
&&\!\!\!\!\!\!\!\!\!\!\!\!\!\!\!\!\!\!\!\!\!\!\!\!\!\!\!\!\!\!\!\!\!\!\!\!\!\!\!\!\!\!\!\!=\sum_{p\leq x}\frac{\log p}{p}\sum_{n\leq x/p}\frac{\Lambda_{k}(n)A_r(mnp)}{n}+O(d_r(m)F_{\tau_0}(m)(\log x)^{k}).
\end{eqnarray*}
Using the induction hypothesis and the prime number theorem, the main term is
\begin{eqnarray*}
&=&\frac{r^{k}}{k!}\sum_{p\leq x}\frac{(\log p)A_r(mp)}{p}\bigg(\log\frac{x}{p}\bigg)^{k}+O(d_r(m)F_{\tau_0}(m)(\log x)^{k})\nonumber\\
&=&\frac{r^{k+1}A_r(m)(\log x)^{k+1}}{(k+1)!}+O(d_r(m)F_{\tau_0}(m)(\log x)^{k}).
\end{eqnarray*}
This completes the proof for \eqref{4}.

Now using \eqref{4} in \eqref{5} and the prime number theorem we deduce that for $k\geq2$
\begin{eqnarray*}
T_{k}(x)&=&\frac{r^{k-1}}{(k-1)!}\sum_{p\leq x}\frac{(\log p)^2A_r(p)}{p}\bigg(\log \frac{x}{p}\bigg)^{k-1}+O((\log x)^{k})\nonumber\\
&=&\frac{r^{k}(\log x)^{k+1}}{(k+1)!}+O((\log x)^{k}).
\end{eqnarray*}
The proof of the lemma is complete.
\end{proof}

\begin{lemm}
Uniformly in $k$ we have
\begin{equation*}
\sum_{n\leq x}\frac{\alpha_k(mn)d_r(n)}{n}\ll \frac{r^k(\log m)^{k+1}(\log x)^{k+1}}{(k+1)!}.
\end{equation*}
As a consequence, for any fixed $\tau_0>0$, we have 
\begin{equation*}
\sum_{n\leq x}\frac{\alpha_k(n)d_r(n)F_{\tau_0}(n)}{n}\ll (\log x)^{k+1}.
\end{equation*}
\end{lemm}
\begin{proof}
The arguments of the first statement is similar to those of the previous lemma. For the second statement we have
\begin{equation*}
F_{\tau_0}(n)\leq\prod_{p|n}(1+Ap^{-\tau_0})=\sum_{d|n}d^{-\tau_0}A^{w(d)},
\end{equation*}
for some $A>0$ and where $w(d)$ is the number of prime factors of $d$. Hence
\begin{eqnarray*}
\sum_{n\leq x}\frac{\alpha_r(n)d_{r}(n)F_{\tau_0}(n)}{n}&\ll&\sum_{n\leq x}\frac{A^{w(n)}}{n^{1+\tau_0}}\sum_{j\leq x/n}\frac{\alpha_r(jn)d_{r}(jn)}{j}\nonumber\\
&\ll&\frac{r^k(\log x)^{k+1}}{(k+1)!}\sum_{n\leq x}\frac{A^{w(n)}d_{r}(n)(\log n)^{k+1}}{n^{1+\tau_0}}\nonumber\\
&\ll&(\log x)^{k+1},
\end{eqnarray*}
since $A^{w(n)}d_{r}(n)(r\log n)^k/(k+1)!\ll n^{\tau_0/2}$ for sufficiently large $n$.
\end{proof}

\section{Proof of Theorem 1}

We first consider the large gaps. We are taking $a(n)=d_r(n)$, i.e.
\begin{equation*}
M(s)=\sum_{n\leq y}\frac{d_r(n)f(\frac{\log y/n}{\log y})}{n^s}.
\end{equation*}
Using Cauchy's residue theorem we have
\begin{equation*}
\sum_{T<\gamma_1\leq2T}|M({\scriptstyle{\frac{1}{2}}}+i\gamma_1+it)|^2=\frac{1}{2\pi i}\int_{\mathscr{C}}\frac{\xi''}{\xi'}(s-it)M(s)M(1-s)ds,
\end{equation*}
where $\mathscr{C}$ is the positively oriented rectangle with vertices at $1-a+i(T+t)$, $a+i(T+t)$, $a+i(2T+t)$ and $1-a+i(2T+t)$. Here and throughout the paper $a=1+L^{-1}$. Now for $s$ inside or on $\mathscr{C}$ we have
\begin{equation*}
M(s)\ll y^{1-\sigma}T^\varepsilon.
\end{equation*}
As in [\textbf{\ref{Da}}] (Chapter 17), we can choose $T'$ such that $T+1<T'<T+2$, $T'+t$ is not the ordinate of a zero of $\xi'(s)$ and $(\xi''/\xi')(\sigma+iT')\ll (\log T)^2$, uniformly for $-1\leq\sigma\leq2$. A simple argument using Cauchy's residue theorem then yields that the contribution of the bottom edge of the contour is $\ll yT^\varepsilon$. Similarly, so is that of the top edge. 

Now from the functional equation we have
\begin{equation*}
\frac{\xi''}{\xi'}(1-s-it)=-\frac{\xi''}{\xi'}(s+it).
\end{equation*}
Hence the contribution from the left edge, by substituting $s$ by $1-s$, is
\begin{eqnarray*}
&&\frac{1}{2\pi i}\int_{a-i(T+t)}^{a-i(2T+t)}\frac{\xi''}{\xi'}(1-s-it)M(s)M(1-s)ds\nonumber\\
&=&-\frac{1}{2\pi i}\int_{a-i(T+t)}^{a-i(2T+t)}\frac{\xi''}{\xi'}(s+it)M(s)M(1-s)ds.
\end{eqnarray*}
We note that this is precisely the conjugate of the contribution from the right edge. Thus, up to an error term of size $O(yT^\varepsilon)$,
\begin{equation}\label{1}
\sum_{T<\gamma_1\leq2T}|M({\scriptstyle{\frac{1}{2}}}+i\gamma_1+it)|^2=2\Re{\bigg(\frac{1}{2\pi i}\int_{a+i(T+t)}^{a+i(2T+t)}\frac{\xi''}{\xi'}(s-it)M(s)M(1-s)ds\bigg)}.
\end{equation}
Using Lemma 4, we can write the expression in the above bracket as 
\begin{equation}\label{13}
I_1+I_2+O((1+L^{2}2^{-K})TL^{r^2}),
\end{equation}
where
\begin{equation*}
I_1=\frac{L}{4\pi}\int_{T}^{2T}|M({\scriptstyle{\frac{1}{2}}}+it)|^2dt,
\end{equation*}
and
\begin{equation*}
I_2=\frac{1}{2\pi i}\int_{a+i(T+t)}^{a+i(2T+t)}\sum_{n=1}^{\infty}\frac{a_{K}(n,s-it)}{n^{s-it}}M(s)M(1-s)ds.
\end{equation*}
Note that we have moved the line of integration in $I_1$ to the $1/2$-line with an admissible error of size $O(yT^\varepsilon)$.

Expanding $M(s)$ we have
\begin{equation*}
I_2=\sum_{m,l\leq y}\sum_{n=1}^{\infty}\sum_{k=0}^{K}\frac{d_{r}(m)d_r(l)f[m]f[l]\alpha_k(n)}{ln^{-it}}\frac{1}{2\pi i}\int_{a+i(T+t)}^{a+i(2T+t)}\frac{1}{L(s-it)^k}\bigg(\frac{l}{mn}\bigg)^{s}ds.
\end{equation*}
Here we denote $f(\frac{\log y/m}{\log y})$ by $f[m]$. Using Lemma 5, we can decompose $I_2$ as $J_1+J_2$, where
\begin{equation*}
J_1=\frac{T}{2\pi}\sum_{\substack{m,l\leq y\\l=mn}}\sum_{k=0}^{K}\bigg(\frac{L}{2}\bigg)^{-k}\frac{d_{r}(m)d_r(l)f[m]f[l]\alpha_k(n)}{ln^{-it}},
\end{equation*}
and
\begin{equation*}
J_2\ll\sum_{\substack{m,l\leq y\\l\ne mn}}\sum_{k=0}^{K}\frac{d_{r}(m)d_r(l)|\alpha_k(n)|}{l|\log\frac{l}{mn}|}\bigg(\frac{l}{mn}\bigg)^a.
\end{equation*}

We first treat the error term. We note that
\begin{equation*}
\alpha_{0}(n)=-\Lambda(n)\ll\log n,\ \textrm{and}\ \alpha_k(n)=\Lambda_{k-1}*\Lambda\log(n)\leq(\log n)^{k+1}.
\end{equation*}
Hence
\begin{equation}\label{11}
J_2\ll\sum_{\substack{m,l\leq y\\l\ne mn}}\frac{d_{r}(m)d_r(l)}{l|\log\frac{l}{mn}|}\bigg(\frac{l}{mn}\bigg)^a(\log n)^{K+1}.
\end{equation}
We now separate whether $|\log\frac{l}{mn}|\geq1$ or $|\log\frac{l}{mn}|<1$. The contribution of the terms $|\log\frac{l}{mn}|\geq1$ to the right hand side of \eqref{11} is
\begin{equation*}
\ll\sum_{m,l\leq y}\frac{d_r(m)d_r(l)}{m}\sum_{n=1}^{\infty}\frac{(\log n)^{K+1}}{n^{1+1/L}}\ll yL^{2r}(KL)^{K}.
\end{equation*}
For the remaining terms, let us assume that $l<mn$ (the other case can be done similarly). We write $mn=l+r$, where $1\leq r\ll l$. Then we have $|\log\frac{l}{mn}|\gg r/l$. Hence the contribution of these to the right hand side of \eqref{11} is
\begin{equation*}
\ll\sum_{m,l\leq y}\sum_{1\leq r\ll l}\frac{d_r(m)d_r(l)}{r}L^{K+1}\ll y^2L^{K+2r}.
\end{equation*}
Thus
\begin{equation}\label{6}
J_2\ll yL^{2r}(KL)^{K}+y^2L^{K+2r}.
\end{equation}

Now for the main term we have
\begin{equation*}
J_1=\frac{T}{2\pi}\sum_{n\leq y}\sum_{k=0}^{K}\bigg(\frac{L}{2}\bigg)^{-k}\frac{\alpha_k(n)}{n^{1-it}}\sum_{m\leq y/n}\frac{d_{r}(m)d_r(mn)f[m]f[mn]}{m}.
\end{equation*}
From Lemma 2, using the Stieltjes integration we obtain
\begin{eqnarray*}
\sum_{m\leq y/n}\frac{d_{r}(m)d_r(mn)f[m]f[mn]}{m}&=&\frac{a_rA_r(n)}{\Gamma(r^2)}\int_{1}^{y/n}\frac{(\log u)^{r^2-1}}{u}f[u]f[un]du\nonumber\\
&&\qquad+O(d_r(n)F_{\tau_0(n)}(\log y)^{r^2-1}).
\end{eqnarray*}
On one hand, the contribution of the $O$-term to $J_1$, using Lemma 7, is
\begin{equation}\label{7}
\ll T(\log y)^{r^2-1}\sum_{n\leq y}\sum_{k=0}^{K}\bigg(\frac{L}{2}\bigg)^{-k}\frac{|\alpha_k(n)|d_r(n)F_{\tau_0}(n)}{n}\ll TKL^{r^2}.
\end{equation}
On the other hand, Lemma 6 and the Stieltjes integration yield that the contribution of the main term to $J_1$ is
\begin{eqnarray*}
&&\frac{a_rT}{2\pi\Gamma(r^2)}\int_{1}^{y}\int_{1}^{y/v}\bigg(-(r+1)+\sum_{k=0}^{K}\frac{(2r\log v/L)^k}{k!}\bigg)\frac{(\log u)^{r^2-1}}{uv^{1-it}}f[u]f[uv]dudv\\
&&\qquad\qquad\qquad\qquad\qquad\qquad\qquad+O(TKL^{r^2}).
\end{eqnarray*}
Substituting $u=y^{1-x}$ and $v=y^\eta$ leads to
\begin{eqnarray*}
&&\frac{a_rT(\log y)^{r^2+1}}{2\pi\Gamma(r^2)}\int_{0}^{1}\int_{0}^{x}(1-x)^{r^2-1}y^{it\eta}\bigg(-(r+1)+\sum_{k=0}^{K}\frac{(r\eta)^k}{k!}\bigg)f(x)f(x-\eta)d\eta dx\\
&&\qquad\qquad\qquad\qquad\qquad\qquad\qquad+O(TKL^{r^2}).
\end{eqnarray*}
Hence, combining with \eqref{6} and \eqref{7} we have
\begin{eqnarray*}
I_2&=&\frac{a_rT(\log y)^{r^2+1}}{2\pi\Gamma(r^2)}\int_{0}^{1}\int_{0}^{x}(1-x)^{r^2-1}y^{it\eta}\big[\exp(r\eta)-(r+1)\big]f(x)f(x-\eta)d\eta dx\\
&&\qquad+O(yL^{2r}(KL)^{K})+O(y^2L^{K+2r})+O(TKL^{r^2})+O(TL^{r^2+1}(3r/K)^{K}).
\end{eqnarray*}
We can ignore the error terms by choosing, for instance, $K=(\log\log T)^2$ and $y=T^{1/2}L^{-K}$. We next take the integration of \eqref{1} from $-\pi\alpha/L$ to $\pi\alpha/L$ and combine with Lemma 3. Simple calculations then give
\begin{equation*} 
h_1(\alpha,M)=\alpha+\frac{2}{\pi}\frac{\int_{0}^{1}\int_{0}^{x}\frac{\sin(\alpha\eta\pi/2)}{\eta}(1-x)^{r^2-1}\big[\exp( r\eta)-r-1\big]f(x)f(x-\eta)d\eta dx}{\int_{0}^{1}(1-x)^{r^2-1}f(x)^2dx}+o(1).
\end{equation*}
With the choice $r=2$ and $f(x)=1+7x-1.5x^2$ we obtain $h_1(1.5)=0.9998$. 

Similarly, for $a(n)=\mu_r(n)$, we have
\begin{equation*}
h_1(\alpha,M)=\alpha+\frac{2}{\pi}\frac{\int_{0}^{1}\int_{0}^{x}\frac{\sin(\alpha\eta\pi/2)}{\eta}(1-x)^{r^2-1}\big[\exp(- r\eta)+r-1\big]f(x)f(x-\eta)d\eta dx}{\int_{0}^{1}(1-x)^{r^2-1}f(x)^2dx}.
\end{equation*}
The choice $r=2$ and $f(x)=1+4.4x+2.3x^2$ yields $h_1(0.7203)=1.000002$. 

This is precisely what we require in \eqref{8}. The proof of the theorem is complete.

\section{Proof of Theorem 2}

In this section, we choose the coefficients $a(n)$ to be supported on $1$ and the primes: $a(1)=a(p)=1$ for $p$ prime. That is to say the mollifier we take have the form
\begin{equation*}
M_1(s)=1+\sum_{p\leq y}\frac{f[p]}{p^s}.
\end{equation*}
In fact, the results obtained are better if we consider the ``twisted'' mollifier
\begin{equation*}
M(s)=M_1(s)+M_1(1-s).
\end{equation*}
The effect of this kind of twist has been observed in [\textbf{\ref{So1}},\textbf{\ref{So}}].

We first consider the denominator of $h_1(\alpha,M)$. We have
\begin{equation*}
\int_{T}^{2T}|M({\scriptstyle{\frac{1}{2}}}+it)|^2dt=2\int_{T}^{2T}|M_1({\scriptstyle{\frac{1}{2}}}+it)|^2dt+2\Re\bigg(\int_{T}^{2T}M_1({\scriptstyle{\frac{1}{2}}}+it)^2dt\bigg).
\end{equation*}
By the Montgomery \& Vaughan's mean value theorem [\textbf{\ref{MV1}}] we obtain
\begin{equation*}
\int_{T}^{2T}M_1({\scriptstyle{\frac{1}{2}}}+it)^2dt\sim T,
\end{equation*}
and
\begin{equation*}
\int_{T}^{2T}|M_1({\scriptstyle{\frac{1}{2}}}+it)|^2dt\sim T\bigg(1+\sum_{p\leq y}\frac{f[p]^2}{p}\bigg).
\end{equation*}
Hence
\begin{equation}\label{20}
\int_{T}^{2T}|M({\scriptstyle{\frac{1}{2}}}+it)|^2dt\sim T\bigg(4+2\sum_{p\leq y}\frac{f[p]^2}{p}\bigg).
\end{equation}

For the nominator, as in \eqref{13}, we have
\begin{equation}\label{18}
\sum_{T<\gamma_1\leq2T}|M({\scriptstyle{\frac{1}{2}}}+i\gamma_1+it)|^2=2\Re{(J_1+J_2)}+O((1+L^{2}2^{-K})TL^{r^2}),
\end{equation}
where
\begin{equation}\label{14}
J_1=\frac{L}{4\pi}\int_{T}^{2T}|M({\scriptstyle{\frac{1}{2}}}+it)|^2dt,
\end{equation}
and
\begin{equation*}
J_2=\frac{1}{2\pi i}\int_{a+i(T+t)}^{a+i(2T+t)}\sum_{n=1}^{\infty}\frac{a_{K}(n,s-it)}{n^{s-it}}M(s)M(1-s)ds.
\end{equation*} 
Since
\begin{equation*}
M(s)M(1-s)=M_1(s)^2+2M_1(s)M_1(1-s)+M_1(1-s)^2,
\end{equation*}
we proceed by writing, say, $J_2=J_{21}+J_{22}+J_{23}$.

As in the previous section, noting that $\alpha_k(1)=0$, we obtain
\begin{equation}\label{15}
J_{21}\ll (KL)^{K}.
\end{equation}
Also
\begin{equation*}
J_{22}=\frac{2T}{2\pi}\sum_{p\leq y}\sum_{k=0}^{K}\bigg(\frac{L}{2}\bigg)^{-k}\frac{f[p]\alpha_{k}(p)}{p^{1-it}}+O(y(KL)^{K+1}+y^2L^K).
\end{equation*}
We note that
\begin{equation*}
\alpha_k(p)=\left\{ \begin{array}{ll}
-\log p &\qquad \textrm{if $k=0$}\\
(\log p)^2 &\qquad \textrm{if $k=1$}\\
0 &\qquad \textrm{if $k\geq2$.}
\end{array} \right.
\end{equation*}
So
\begin{equation}\label{16}
J_{22}=-\frac{2T}{2\pi}\sum_{p\leq y}\frac{\log p(1-2\log p/L)f[p]}{p^{1-it}}+O(y(KL)^{K+1}+y^2L^K).
\end{equation}
Similarly we have
\begin{eqnarray*}
J_{23}&=&\frac{2T}{2\pi}\sum_{p\leq y}\sum_{k=0}^{K}\bigg(\frac{L}{2}\bigg)^{-k}\frac{f[p]\alpha_{k}(p)}{p^{1-it}}+\frac{T}{2\pi}\sum_{p,q\leq y}\sum_{k=0}^{K}\bigg(\frac{L}{2}\bigg)^{-k}\frac{f[p]f[q]\alpha_{k}(pq)}{(pq)^{1-it}}\\
&&\qquad\qquad+O((KL)^{K+1}+y^2L^K).
\end{eqnarray*}
Now the first term is precisely the main term of $J_{22}$. Furthermore, it is standard to verify that the second term is
\begin{equation}\label{17}
\frac{4T}{2\pi}\sum_{p\ne q\leq y}\frac{\log p\log q(\log p+\log q)f[p]f[q]}{L^2(pq)^{1-it}}+O(T).
\end{equation}

We next take the integration of \eqref{18} from $-\pi\alpha/L$ to $\pi\alpha/L$. Combining \eqref{20}, \eqref{14}--\eqref{17} and ignoring the error terms (by choosing some admissible $K$ and $y$ as before) we easily obtain
\begin{equation*}
h_1(\alpha,M)=\alpha+g_1(\alpha)+g_2(\alpha)+o(1),
\end{equation*}
where
\begin{equation*}
g_1(\alpha)=-\frac{4}{\pi}\frac{\sum_{p\leq y}(1-2\log p/L)\sin(\frac{\pi\alpha\log p}{L})f[p]/p}{2+\sum_{p\leq y}f[p]^2/p},
\end{equation*}
and
\begin{equation*}
g_2(\alpha)=\frac{4}{\pi}\frac{\sum_{p,q\leq y}\log p\log q\sin(\frac{\pi\alpha\log (pq)}{L})f[p]f[q]/(L^2pq)}{2+\sum_{p\leq y}f[p]^2/p},
\end{equation*}
We now choose 
\begin{equation*}
f[p]=-c\bigg(1-\frac{2\log p}{L}\bigg)\sin\bigg(\frac{\pi\alpha\log p}{L}\bigg),
\end{equation*}
where $c$ is some constant which we will specify later. Then from the prime number theorem and the Stieltjes integration we have
\begin{equation*}
h_1(\alpha,M)=h_1(\alpha,c)=\alpha+\frac{4Uc+Vc^2}{\pi(2+Uc^2)}+o(1),
\end{equation*}
where
\begin{equation*}
U=\int_{0}^{1}\frac{(1-u)^2\sin^2(\frac{\pi\alpha u}{2})du}{u},
\end{equation*}
and
\begin{equation*}
V=\int_{0}^{1}\int_{0}^{1}(1-u)(1-v)\sin\big(\frac{\pi\alpha u}{2}\big)\sin\big(\frac{\pi\alpha v}{2}\big)\sin\big(\frac{\pi\alpha (u+v)}{2}\big)dudv.
\end{equation*}
The optimal choice of $c$ will then be
\begin{equation*}
c_{\pm}=\frac{V\pm\sqrt{V^2+8U^3}}{2U^2}.
\end{equation*}
With the help of Maple, we can verify that $h_1(1.18,c_{-})=0.9995$ and $h_1(0.796,c_{+})=1.00006$. 

We are left to prove that
\begin{equation}\label{19}
\frac{\big(\int_{T}^{2T}|M({\scriptstyle{\frac{1}{2}}}+it)|^{2}dt\big)^4}{\big(\sum_{T<\gamma_1\leq2T}\delta(\gamma_1)^2\big)\big(\int_{T}^{2T}|M({\scriptstyle{\frac{1}{2}}}+it)|^{4}dt\big)^2}\gg TL^{-1},
\end{equation}
and
\begin{equation}\label{12}
\frac{\big(\int_{T}^{2T}|M({\scriptstyle{\frac{1}{2}}}+it)|^{2}dt\big)^2}{h_2(\mu,M)\int_{T}^{2T}|M({\scriptstyle{\frac{1}{2}}}+it)|^{4}dt}\gg T.
\end{equation}
Following the arguments of Fujii [\textbf{\ref{F}}] one can show that
\begin{equation*}
\sum_{T<\gamma_1\leq2T}\delta(\gamma_1)^2\ll TL.
\end{equation*}
The estimates \eqref{19} and \eqref{12} now just follow from [\textbf{\ref{CGGH}}] (see (16) and (17)). This completes the proof of the theorem.

\end{document}